\documentclass[10pt,reqno]{amsart}

\usepackage{amssymb}
\usepackage{amsfonts}
\usepackage{amsthm}
\usepackage{amsmath}
\usepackage{bbm}

\numberwithin{equation}{section}
\allowdisplaybreaks

\newtheorem{theorem}{Theorem}[section]
\newtheorem{proposition}[theorem]{Proposition}

\newtheorem{lemma}[theorem]{Lemma}
\newtheorem{corollary}[theorem]{Corollary}

\theoremstyle{definition}

\theoremstyle{remark}

\newcommand{\R}{\mathbb{R}}

\newcommand{\N}{\mathbb{N}}
\newcommand{\Z}{\mathbb{Z}}

\renewcommand{\hat}{\widehat}
\newcommand{\eps}{\varepsilon}

\newcommand{\scriptD}{\mathcal{D}}
\newcommand{\scriptE}{\mathcal{E}}

\newcommand{\scriptL}{\mathcal{L}}

\newcommand{\scriptR}{\mathcal{R}}
\newcommand{\scriptS}{\mathcal{S}}

\newcommand{\qtq}[1]{\quad\text{#1}\quad}

\DeclareMathOperator*{\wklim}{wk-lim}
\DeclareMathOperator*{\supp}{supp}

\DeclareMathOperator*{\dist}{dist}

 \usepackage[shortlabels]{enumitem}
                    \setlist[enumerate, 1]{1\textsuperscript{o}}

\begin{document}
\title{Extremizability of Fourier restriction to the paraboloid}
\author{Betsy Stovall}
\email{stovall@math.wisc.edu}
\address{480 Lincoln Drive, Madison, WI 53706}
\begin{abstract}
In this article, we prove that nearly all valid, scale-invariant Fourier restriction inequalities for the paraboloid in $\R^{1+d}$ have extremizers and that $L^p$-normalized extremizing sequences are precompact modulo symmetries.  This result had previously been established for the case $q=2$.  In the range where the boundedness of the restriction operator is still an open question, our result is conditional on improvements toward the restriction conjecture.  
\end{abstract}

\maketitle

\section{Introduction}

There has been substantial recent attention paid to the problem of determining whether equality is possible for various inequalities in harmonic analysis.  In the case of Fourier restriction inequalities, essentially all of the known results are for $L^2$-based Fourier restriction, wherein it is possible to use Hilbert space techniques and Plancherel; an excellent survey of recent results in this vein may be found in \cite{FoschiSilva}.  An exception is a result of Christ--Quilodr\'an \cite{ChristQuilodran} stating that Gaussian functions are not maximizers for Fourier restriction to the paraboloid except in the $L^1$ case and possibly in the Stein--Tomas--Strichartz case.  The result of \cite{ChristQuilodran}, however, leaves open the question of whether maximizers actually exist for the intermediate Lebesgue space bounds.  The purpose of this article is to establish the existence of extremizers and precompactness of extremizing sequences for all valid, nonendpoint, $L^p$ to $L^q$ restriction inequalities for the paraboloid, including, conditionally, the conjectural ones.  We note that the existence of a second endpoint restriction inequality, i.e.\ other than trivial one at $L^1$, would be rather a surprise to the harmonic analysis community, and thus it is expected that our result is sharp.  

We start with a quick recap of the current state of the restriction problem, which will give us an opportunity to define the notation and terminology needed to state our results.

In the late 1960s, Stein conjectured that the restriction operator
$$
\scriptR f(\xi) := \hat f(|\xi|^2,\xi)
$$
extends as a bounded operator from  $L^p(\R^{1+d})$ into $L^q(\R^d)$ for all pairs $(p,q)$ satisfying
\begin{equation} \label{E:restriction admissible} 
p=\frac{d q'}{d+2} \qtq{and} q > p.
\end{equation}
An equivalent formulation is that the extension operator
$$
\scriptE f(t,x) = \int_{\R^d} e^{i(t,x)(|\xi|^2,\xi)}f(\xi)\, d\xi,
$$
extends as a bounded operator from $L^p(\R^d)$ to $L^q(\R^{1+d})$ for all pairs $(p,q)$ satisfying
\begin{equation} \label{E:extension admissible}
 q=\tfrac{d+2}d p' \qtq{and} q > p.  
 \end{equation}
Conditions \eqref{E:restriction admissible} and \eqref{E:extension admissible} are known to be necessary for boundedness of $\scriptR$ and $\scriptE$, respectively.  As of this writing, the above described restriction/extension conjecture is settled when $d=1$, by Stein and Fefferman, and is open in all higher dimensions.  More precisely, in higher dimensions, it is solved for $q > 3.25$ (\cite{Guth_d=2, BassamShayya}; see also \cite{HWang}) when $d=2$, and for $q > q(d)$, for some (explicit, yet complicated) $q(d) < \frac{2(d+3)}{d+1}$ when $d \geq 3$ \cite{Guth_dgeq3, HickmanRogers, TaoParab, TVV}.  

Classical symmetries of both the Fourier transform and the paraboloid lead to a wealth of symmetries for the extension and restriction operators.  These symmetries, in turn, are of paramount importance in the study of uniqueness and compactness questions for maximizers and near maximizers of the Lebesgue bounds for $\scriptE$ and $\scriptR$.  

To be more precise, by a symmetry of the extension operator $\scriptE:L^p_\xi \to L^q_{t,x}$, we mean a pair $(S,T)$ with $S$ an isometry of $L^p_\xi$, $T$ an isometry of $L^q_{t,x}$, and $\scriptE \circ S = T \circ \scriptE$.  We let $\scriptS_p$ denote the group of all symmetries of $\scriptE:L^p_\xi \to L^q_{t,x}$.  For simplicity, we will often abuse notation by associating the symmetry $(S,T)$ with its first coordinate, $S$.    Under this convention, $\scriptS_p$ contains the dilations, $f \mapsto \lambda^{\frac dp} f(\lambda \cdot)$, the frequency translations, $f \mapsto f(\cdot-\xi_0)$, the space-time modulations, $f \mapsto e^{i(t_0,x_0)(|\cdot|^2,\cdot)}f$, and compositions of these three.  There are other symmetries, such as rotations and multiplication by unimodular constants, but these generate compact subgroups of $\scriptS_p$, and therefore play no role in our analysis.  We let $\tilde \scriptS_p$ denote the subgroup of $\scriptS_p$ generated by the aforementioned noncompact symmetry groups.  Finally, we note that if $(S,T)$ is a symmetry of $\scriptE:L^p_{\xi} \to L^q_{t,x}$, then $(T^*, S^*)$ is a symmetry of the corresponding restriction operator, $\scriptR:L^{q'}_{t,x} \to L^{p'}_\xi$.

Fix a pair $(p,q)$ for which the extension operator is bounded, and let $A_p:=\|\scriptE\|_{L^p_\xi \to L^q_{t,x}}$.  In this article we take up two natural questions:  Do there exist nonzero functions that achieve equality in the estimates
\begin{equation} \label{E:extn and rest}
\|\scriptE f\|_q \leq A_p\|f\|_p, \qtq{and} \|\scriptR g\|_{p'} \leq A_p \|g\|_{q'}?,
\end{equation}
and, Must a function that nearly achieves equality be close to one that achieves equality?  Under the additional condition that there exists an exponent pair $(\tilde p,\tilde q)$ satisfying $\tilde p > p$ at which $\scriptE$ is bounded, we answer both of these questions in the affirmative, and show in addition that the intersection of the $L^p_\xi$ (resp., $L^{q'}_{t,x}$) unit sphere with the set of all $f$ (resp.\ $g$) achieving equality in \eqref{E:extn and rest} is compact modulo symmetries.  

To state our result more precisely, we will call a nonzero $L^p_\xi$ function $f$ an (extension) \textit{extremizer} if it achieves equality in \eqref{E:extn and rest} and a nonzero $L^p_\xi$ sequence $\{f_n\}$ \textit{extremizing} if $\lim\frac{\|\scriptE f_n\|_q}{\|f_n\|_p} = A_p$; restriction extremizers are defined analogously.  For questions of compactness, it is most natural to work with \textit{$L^p_\xi$-normalized extremizing sequences}, that is, extremizing sequences $\{f_n\}$ with $\|f_n\|_p \equiv 1$.  

\begin{theorem} \label{T:main}
Assume that the extension operator extends as a bounded operator from $L^{p_0}_\xi$ to $L^{q_0}_{t,x}$ for some $1 < p_0 < q_0 = \tfrac{d+2}d p_0'$.  Let $1 < p < p_0$ and $q=\tfrac{d+2}dp'$.  Define $A_p$ to be the $L^p_\xi \to L^q_{t,x}$ operator norm of $\scriptE$.  If $\{f_n\} \subseteq L^p_\xi$ is an $L^p_\xi$-normalized extremizing sequence, then, after passing to a subsequence, there exist symmetries $\{S_n\} \subseteq \tilde\scriptS_p$ such that $S_n f_n \to f$, in $L^p_\xi$, for some extremizer $f$ of the extension inequality in \eqref{E:extn and rest}.  
\end{theorem}

By duality and strict convexity of $L^p$ when $1 < p < \infty$, we have the analogous result for the restriction operator.  Details of this deduction are given in Section~\ref{S:corollary}.

\begin{corollary}\label{C:cor}
Assume that the restriction operator extends as a bounded operator from $L^{p_0}_{t,x}$ to $L^{q_0}_\xi$, for some $\infty > q_0 > p_0 = (\tfrac{d+2}d q)'$.  Let $1 < p < p_0$ and $q=\tfrac d{d+2}p'$.  Then any $L^p_{t,x}$-normalized extremizing sequence for $\scriptR$ has a subsequence that, modulo symmetries, converges in $L^p_{t,x}$ to an extremizer of the restriction inequality in \eqref{E:extn and rest}.
\end{corollary}

That $\scriptE:L^1_\xi \to L^\infty_{t,x}$ and $\scriptR:L^1_{t,x} \to L^\infty_\xi$ possess extremizers and that extremizing sequences for these operators need not be compact are both elementary; indeed, for both, one need only consider sequences of the form $\{\phi + \phi(\cdot + ne_1)\}$, for some $0 \leq \phi \in L^1$ with $\phi \not\equiv 0$.  

The Stein--Tomas--Strichartz case $p=2$ has been well-studied, and extremizers are known to exist in all dimensions \cite{BennettBezCarberyHundertmark, Foschi07, HundertmarkZharnitsky, ShaoParab}.  It is conjectured that radial Gaussians are, up to symmetries, the unique extremizers of the $L^2_\xi \to L^{\frac{2(d+2)}d}_{t,x}$ inequality, but this is only known in dimensions $d=1,2$ \cite{Foschi07}, wherein the exponent $q=\frac{d+2}d 2'$ is an even integer.  Curiously, it is known in all dimensions that radial Gaussians are not extremal for any $L^p_\xi \to L^q_{t,x}$ inequality for $\scriptE$ unless $p \in \{1,2\}$ \cite{ChristQuilodran}, but outside of the cases $p=1,2$, extremizers had not been previously shown to exist.  

Unfortunately, the aforementioned proofs of existence of extremizers to the Stein--Tomas--Strichartz inequality rely on Plancherel and the Hilbert space structure of $L^2_\xi$.  The proof that seems most amenable to generalization is that of \cite{ShaoParab} (see also \cite{Lions}), which applies the linear profile decomposition of \cite{BV, CarlesKeraani, MerleVega}. 

We turn to a heuristic overview of the statement and proof of the $L^2_\xi$-based linear profile decomposition in general dimensions.  (We give a formal statement later.)  This overview will allow us to explain some of the difficulties in adapting the $L^2_\xi$-based arguments to the general case.   Set $q_2:= \frac{2(d+2)}d$.  By Tao's bilinear extension estimate for the paraboloid \cite{TaoParab} and an adaptation of the bilinear to linear argument of Tao--Vargas--Vega \cite{TVV}, one can prove an ``improved Strichartz inequality,''  which implies that if $\{f_n\}$ is an $L^2_\xi$-normalized sequence with $\|\scriptE f_n\|_{q_2} \not\to 0$, then there is a nontrivial contribution coming from a portion of $f_n$ well-adapted to a ball \cite{BV}.  After passing to a subsequence and applying a suitable sequence $\{S_n\}$ of symmetries, there is a nontrivial weak limit:  $S_n^{-1}f_n \rightharpoonup f \neq 0$.  For large $n$, $f_n-S_n f_n$ has a smaller $L^2_\xi$ norm, via arguments that are elementary in Hilbert spaces.  By repeating this process, we can, after passing to a subsequence, write the extension $\scriptE f_n$ as a sum of a finite number of asymptotically (pairwise) orthogonal profiles, together with an error that is small in $L^{q_2}_{t,x}$.  One can show, either directly \cite{CarlesKeraani, MerleVega} or by using local smoothing estimates \cite{ClayNotes}, that the $L^{q_2}_{t,x}$ norms of  asymptotically orthogonal bubbles decouple.  Thus, after passing to a subsequence, for each $J \geq 1$ we may decompose $f_n = \sum_{j=1}^J S_n^j \phi^j + w_n^J$, where 
\begin{equation} \label{E:both decoup}
\begin{gathered}
\lim_{n \to \infty} \|f_n\|_2^2 - \bigl(\sum_j \|\phi_j\|_2^2+\|w_n^J\|_2^2\bigr) = 0 \\
 \lim_{n \to \infty} \|\scriptE f_n\|_{q_2}^{q_2} - \bigl( \sum_j \|\scriptE \phi_j\|_{q_2}^{q_2}+\|w_n^J\|_{q_2}^{q_2}\bigr) = 0,
\end{gathered}
\end{equation}
and, moreover, $\lim_{J \to \infty} \limsup_{n \to \infty} \|\scriptE w_n^J\|_{q_2} \to 0$, 

If the sequence $f_n$ is extremizing, strict convexity (coming from $2 < q_2$) dictates that in fact there is only one profile and the error $w_n$ tends to zero in $L^2_\xi$ \cite{ShaoParab}.  

Unfortunately, for any $p \neq  2$, we hit a snag early on, because it is possible to increase the limit of the norms of a sequence in $L^p_\xi$ by subtracting the weak limit of the sequence; a simple example is given in \cite{OpialBulletin67}.  It is natural to try to subtract a ``positive'' quantity from the sequence to reduce the $L^p_\xi$ norm (such as the portion of each $f_n$ that is well-adapted to a ball), but this presents some challenges since $\scriptE$ is not a positive operator; in particular, we must keep the spacetime modulations under control, and in order to use convexity, it is crucial that we have equality in both estimates in \eqref{E:both decoup}.  The need for this precision is also why the general framework for profile decompositions in Banach spaces found in \cite{SoliminiTintarev} does not seem to directly yield our result.  

Our approach is to first control the positive symmetries: the dilations and frequency translations.  To gain this control, we generalize the improved Strichartz inequality of \cite{BV, CarlesKeraani, MerleVega} to $L^p_\xi$ (Lemma~\ref{L:big bite}).  This inequality controls the extension with a nontrivial positive operator \eqref{E:extn^2}, whose norm can be reduced by deleting the portions of each $f_n$ well-adapted to balls.  The portion of each $f_n$ which contributes to $\scriptE f_n$ is, by convexity, carried on a controllable number of balls (Proposition~\ref{P:freq nibbles}), and we then use convexity again to show that the major portion of each $f_n$ is, in fact, carried on a single ball (Proposition~\ref{P:freq localization}).  Applying a symmetry, we then have an extremizing sequence in $L^p_\xi$ that is nearly bounded with compact support, which means that it is almost in $L^2_\xi$.  (This part of the argument is closely related to Lieb's \textit{method of missing mass} \cite{Lieb}; see also \cite{FLS}.)

Truncating introduces some error, but lets us apply the $L^2_\xi$-based profile decomposition.  The only profiles that can arise are spacetime translates, because the compact support and boundedness of our truncations rule out dilations and frequency translations.  In order to reduce both the $L^p_\xi$ and the $L^q_{t,x}$ norm, we need to take care with how we extract $L^p_\xi$ bubbles from this decomposition; we do this by carefully truncating on the spacetime side and bounding a vector-valued operator (Lemma~\ref{L:Schrodinger projections}).  Finally, this yields an $L^p_\xi$-based profile decomposition (Proposition~\ref{P:Lp profile}) in which both the $L^p_\xi$ norms of the profiles and the $L^q_{t,x}$ norms of the extensions are sufficiently decoupled that we can, in Section~\ref{S:prove profile}, prove Theorem~\ref{T:main}.

\subsection*{Terminology}
For nonnegative numbers $A$ and $B$, we will write $A \lesssim B$ to mean that $A \leq C B$ for a constant $C$ that depends only on $d,p_0,p,A_{p_0}$ but that is otherwise allowed to change from line to line.  A dyadic interval is an interval of the form $[m2^{-n},(m+1)2^{-n}]$, $m,n \in \Z$, and a dyadic cube is a product of dyadic intervals all having the same length.  We denote the set of all dyadic cubes of side length $2^{-k}$ by $\scriptD_k$, and an individual dyadic cube will typically be denoted $\tau$.  To simplify later statements, we will consider the empty set to be a dyadic cube, $\emptyset \in \scriptD_{\infty}$.  We will also use the little `o' notation; $o_R(1)$ will denote a quantity that tends to zero as $R \to \infty$.  

\subsection*{Acknowledgements}
The author would like to thank Mike Christ and Taryn Flock for valuable conversations in the early stages of this project.  These discussions helped inspire the eventual idea for the proof.  She would also like to thank the anonymous referee for a number of thoughtful and extremely helpful comments.  This work was supported in part by NSF DMS-1600458; in addition, part of this research was supported by NSF DMS-1440140, while the author was in residence at the Mathematical Sciences Research Institute in Berkeley, California, during the Spring Semester of 2017.


\section{Positive profiles}


In this section, we prove the following proposition, which allows us to nibble away at the absolute value of a function, while reducing its extension in a quantitative way.  

\begin{proposition} \label{P:freq nibbles}
There exist a sequence $\rho_j \searrow 0$ such that for every $f \in L^p$, there exists a sequence $\tau^j$ of dyadic cubes such that if the sequences $g^j,r^j$ are defined inductively by setting 
$$
r^0:=f, \qquad g^j:=r^{j-1}\chi_{\tau^j}\chi_{\{|f|<\rho_j^{-1}\|f\|_p|\tau^j|^{-\frac1p}\}}, \qquad r^j:=r^{j-1}-g^j, 
$$
then $\|\scriptE h\|_q \leq \rho_j \|f\|_p$ for all $j$ and all measurable functions $|h| \leq |r^j|$.  
\end{proposition}

The main step in proving the proposition is to prove the following lemma.

\begin{lemma} \label{L:big bite}
Let $f \in L^p$.  For some $0 < \theta < 1$ and $c_0 > 0$,
\begin{equation} \label{E:big bite}
\|\scriptE f\|_q \lesssim \sup_{k \in \Z} \sup_{\tau \in \scriptD_k} \sup_{l \geq 0} 2^{-c_0l}\|f_{\tau,l}\|_p^\theta \|f\|_p^{1-\theta},
\end{equation}
where $f_{\tau,l}$ equals $f$ multiplied by the characteristic function of 
$$
\tau \cap \{|f|<2^l\|f\|_p |\tau|^{-\frac1p}\}.
$$  
\end{lemma}

Assuming the lemma for a moment, we give the short proof of Proposition~\ref{P:freq nibbles}.

\begin{proof}[Proof of Proposition~\ref{P:freq nibbles}]
Given $0 \neq f \in L^p$ and a sequence $\{\rho_j\} \subseteq (0,\infty)$, we may, by a routine application of the Dominated Convergence Theorem, select dyadic cubes $\tau^j$ so that 
$$
\|g^j\|_p = \max_{\tau\,\text{dyadic}}\|r^{j-1} \chi_\tau \chi_{\{|f|<\rho_j^{-1}\|f\|_p|\tau|^{-\frac1p}\}}\|_p,
$$
in the notation of the proposition.  We will inductively construct a sequence $\rho_j \searrow 0$ such that the conclusion of the proposition holds for every $f$ with this choice of dyadic cubes.  

We assume that we are given $k \geq 0$ and integers $0=J_0 < \cdots < J_k$ such that the conclusions of the proposition hold for the sequence 
\begin{equation} \label{E:def rho_J}
\rho_{J_i+1} := \cdots := \rho_{J_{i+1}} := 2^{-i}A_p, \: i < k, \qtq{and} \rho_j:=2^{-k}A_p, \quad j > J_k
\end{equation}
and corresponding dyadic cubes.  (In the base case, $k=0$, the assumption follows from the definition of $A_p$.)  

Let $J_{k+1}>J_k$ be a sufficiently large integer, to be determined in a moment, and consider a function $f$ with $\|f\|_p = 1$.  By monotonicity of the remainder terms, the inductive step will be complete once we prove that for $J_{k+1}$ sufficiently large, the extension of a measurable $|h| \leq |r^{J_{k+1}}|$ obeys a better bound, $\|\scriptE h\|_q \leq 2^{-(k+1)}A_p$.  

By the maximality property of our cubes,
$$
\max_\tau \|r^{J_{k+1}} \chi_\tau \chi_{\{|f|<\rho_{J_{k+1}}^{-1} |\tau|^{-\frac1p}\}}\|_p \leq \min_{J_k < j \leq J_{k+1}}\|g^j\|_p.
$$
By construction, the $g_j$ have pairwise disjoint supports and $\sum_j |g_j| \leq |f|$.  Therefore
$$
(J_{k+1}-J_k)^{\frac1p} \min_{J_k < j \leq J_{k+1}}\|g_j\|_p \leq \bigl(\sum_{j=J_k+1}^{J_{k+1}}\|g_j\|_p^p\bigr)^{\frac1p} \leq \|f\|_p,
$$
so
\begin{equation} \label{E:max tau rJ}
\max_\tau \|r^{J_{k+1}} \chi_\tau \chi_{\{|f|<\rho_{J_{k+1}}^{-1} |\tau|^{-\frac1p}\}}\|_p \leq (J_{k+1}-J_k)^{-\frac1p}.
\end{equation}

Now let $L = L_k$ be an integer, sufficiently large that $B 2^{-c_0L} \leq 2^{-(k+1)}A_p$, with $B$ the implicit constant in \eqref{E:big bite}.  We may assume that $2^{-L} \leq 2^{-k}A_p = \rho_{J_{k+1}}$.  Let $|h| \leq |r^{J_{k+1}}|$ be a measurable function.  By \eqref{E:big bite}, 
$$
\|\scriptE h\|_q \leq B\max\{2^{-c_0L}, \max_{\tau\,\text{dyadic}} \|h \chi_\tau \chi_{\{|f|<2^{-L}|\tau|^{-\frac1p}\}}\|_p^\theta\}.
$$
Each dyadic $\tau$ is covered by at most $C(2^L\rho_{J_{k+1}}^{-1})^{pd}$ dyadic $\tau'$ whose diameter equals $(2^{-L}\rho_{J_{k+1}})^p$ times the diameter of $\tau$, so
$$
\max_{\tau\,\text{dyadic}} \|h \chi_\tau \chi_{\{|f|<2^{-L}|\tau|^{-\frac1p}\}}\|_p \lesssim (2^L \rho_{J_{k+1}}^{-1})^d\max_{\tau'\,\text{dyadic}}\|h\chi_{\tau'}\chi_{\{|h|<\rho_{J_{k+1}}|\tau'|^{-\frac1p}\}}\|_p.
$$
Thus
$$
\|\scriptE h\|_q \leq \max\{2^{-(k+1)}A_p, C_k (J_{k+1}-J_k)^{-\frac\theta p}\},
$$  
and taking $J_{k+1}$ sufficiently large completes the proof.  
\end{proof}

The proof of Lemma~\ref{L:big bite} is an adaptation of the argument of B\'egout--Vargas \cite{BV}, which was carried out there in the case $p=2$.  See also \cite{Bourgain98, Keraani, MVV} for earlier results in a similar vein.  

\begin{proof}[Proof of Lemma~\ref{L:big bite}]
Throughout the proof, we assume that $\|f\|_p = 1$.  

Tao's bilinear adjoint restriction theorem \cite{TaoParab} for the paraboloid states that
$$
\|\scriptE f_1 \scriptE f_2\|_r \lesssim \|f_1\|_2\|f_2\|_2,
$$
for $r > \frac{d+3}{d+1}$ and $f_1,f_2$ supported on cubes in $\scriptD_0$ that are separated by a distance of at least 1, while any valid $L^p_\xi \to L^q_{t,x}$ linear estimate trivially yields a bilinear estimate
$$
\|\scriptE f_1 \scriptE f_2\|_{\frac{q}2} \leq \|\scriptE f_1\|_{q}\|\scriptE f_2\|_{q} \leq A_{p}\|f_1\|_{p}\|f_2\|_{p}.
$$
Interpolating these two bilinear estimates and rescaling implies that for any nonendpoint pair $(p,q)$ for which \eqref{E:extn and rest} holds, there exists some $s < p$ such that 
\begin{equation} \label{E:Ls bilin}
\|\scriptE f_\tau \scriptE f_{\tau'}\|_{q/2} \lesssim 2^{2kd(\frac1s-\frac1p)}\|f_\tau\|_s\|f_{\tau'}\|_s,
\end{equation}
whenever $f_\tau,f_{\tau'}$ are supported on cubes $\tau,\tau' \in \scriptD_k$ separated by a distance of at least $2^{-k}$.  

Recalling \cite{TVV}, for $\tau, \tau' \in \scriptD_k$, we say that $\tau \sim \tau'$ if $C_d 2^{-k} \leq \dist(\tau,\tau') \leq 2C_d 2^{-k}$, with $C_d$ sufficiently large.  For every $\xi \neq \xi'$, there exist $k$ and $\tau \sim \tau' \in \scriptD_k$ such that $\xi \in \tau$ and $\xi' \in \tau'$.  Thus 
$$
\|\scriptE f\|_q^2 = \|\scriptE f \scriptE f\|_{\frac q2}  = \|\sum_k \sum_{\tau \sim \tau' \in \scriptD_k} \scriptE f_\tau \scriptE f_{\tau'}\|_{\frac q2},
$$
where $f_\tau = \chi_\tau f$, with $\chi_\tau$ a cutoff supported on $\tau$.  
The product $\scriptE f_\tau \scriptE f_{\tau'}$ has frequency support in
$$
\{(|\xi|^2+|\xi'|^2,\xi+\xi') : \xi \in \tau, \xi' \in \tau'\}, 
$$
which, for $C_d$ sufficiently large, is contained in a parallelepiped 
$$
R_{\tau,\tau'} \subseteq \{(\eta,\zeta): \zeta \in \tau+\tau', |\eta-\tfrac12|\zeta|^2| \sim 2^{-2k}\}.
$$
These parallelepipeds are finitely overlapping as $k$ and $\tau \sim \tau'$ vary, so by the almost orthogonality lemma (Lemma~6.1) of \cite{TVV},
$$
\|\scriptE f\|_q^2 \lesssim \bigl(\sum_k \sum_{\tau \sim \tau' \in \scriptD_k} \|\scriptE f_\tau \scriptE f_{\tau'}\|_{\frac q2}^t\bigr)^{\frac1t},
$$
where $t = \min\{\frac q2,(\frac q2)'\}$.  From our bilinear restriction inequality \eqref{E:Ls bilin},
$$
\|\scriptE f_\tau \scriptE f_{\tau'}\|_{\frac q2} \lesssim 2^{2kd(\frac1s-\frac1p)}\|f_\tau\|_s\|f_{\tau'}\|_s\lesssim 2^{2kd(\frac1s-\frac1p)}\|f_{\tau''}\|_s^2,
$$
where $\tau''$ is a slightly larger cube containing both $\tau$ and $\tau'$.  Thus, after reindexing,
\begin{equation} \label{E:extn^2}
\|\scriptE f\|_q^2 \lesssim \bigl(\sum_k \sum_{\tau \in \scriptD_k} 2^{2kdt(\frac1s-\frac1p)}\|f_{\tau}\|_s^{2t}\bigr)^{\frac1t}.
\end{equation}
Arithmetic shows that $2t > p$, and we recall that $p > s$.  

This completes the proof of Lemma~\ref{L:big bite}, modulo the inequality
$$
\bigl(\sum_k \sum_{\tau \in \scriptD_k} |\tau|^{-2t(\frac1s-\frac1p)}\|f_\tau\|_s^{2t}\bigr)^{\frac1{2t}} \lesssim \sup_{k \in \Z} \sup_{\tau \in \scriptD_k} \sup_{l \geq 0} 2^{-c_0l} \|f_{\tau,l}\|_p^\theta \|f\|_p^{1-\theta},
$$
which we take up in the next lemma.
\end{proof}

\begin{lemma} \label{L:Xs2t bound}
Assume that $\|f\|_p = 1$.  
If $2t > p > s$ and $\max\{\frac p{2t},\frac sp\} < \theta < 1$, then for some $c_0 > 0$,
\begin{equation} \label{E:Xs2t bound}
\bigl(\sum_k \sum_{\tau \in \scriptD_k} |\tau|^{-2t(\frac1s-\frac1p)}\|f_\tau\|_s^{2t}\bigr)^{\frac1{2t}} \lesssim \sup_{k \in \Z} \sup_{\tau \in \scriptD_k} \sup_{l \geq 0} 2^{-c_0l} \|f_{\tau,l}\|_p^\theta.
\end{equation}
\end{lemma}


\begin{proof}[Proof of Lemma~\ref{L:Xs2t bound}]
We will prove the slightly stronger (since $|f_\tau^l| \leq |f_{\tau,l}|$) inequality
\begin{equation} \label{E:strong Xs2t bound}
\bigl(\sum_k \sum_{\tau \in \scriptD_k} |\tau|^{-2t(\frac1s-\frac1p)}\|f_\tau\|_s^{2t}\bigr)^{\frac1{2t}} \lesssim \sup_{k \in \Z} \sup_{\tau \in \scriptD_k} \sup_{l \geq 0} 2^{-c_0l} \|f_\tau^l\|_p^\theta,
\end{equation}
where the $f_{\tau}^l$ are defined inductively by $f_\tau^0 := f_{\tau,0}$ and $f_\tau^l := f_{\tau,l} - f_\tau^{l-1}$, $l \geq 1$.  

For any $c_1 > 0$, by using disjointness of the supports of the $f_\tau^l$, then H\"older, then two more applications of H\"older (together with summability of $2^{-\frac{c_1s}2(2t-s)}$),
\begin{align} \notag
&\sum_k \sum_{\tau \in \scriptD_k} |\tau|^{-2t(\frac1s-\frac1p)}\|f_\tau\|_s^{2t} 
\lesssim \sum_k \sum_{\tau \in \scriptD_k} |\tau|^{-2t(\frac1s-\frac1p)}\bigl(\sum_{l \geq 0} \|f_\tau^l\|_s^s\bigr)^{\frac{2t}s}\\ \notag
&\quad \lesssim \sup_k \sup_{\tau \in \scriptD_k} \sup_{l \geq 0} \bigl(2^{-\frac{c_1}{2(1-\theta)} l}|\tau|^{-2t(\frac1s-\frac1p)}\|f_\tau^l\|_s^{2t}\bigr)^{1-\theta} \\\notag
&\qquad\qquad\qquad\qquad\qquad \times \sum_k \sum_{\tau \in \scriptD_k} |\tau|^{-2t\theta(\frac1s - \frac1p)}\bigl(\sum_{l \geq 0} 2^{\frac{c_1s}{4t}l}\|f_\tau^l\|_s^{s\theta}\bigr)^{\frac{2t}s}\\ \label{E:l outside}
& \quad \lesssim \sup_{k \in \Z} \sup_{\tau \in \scriptD_k}\sup_{l \geq 0} \bigl(2^{-c_0l}\|f_\tau^l\|_p^{2t}\bigr)^{1-\theta} \bigl(\sum_{l \geq 0} 2^{c_1l} \sum_{k \in \Z} 2^{2kdt\theta(\frac1s-\frac1p)}\sum_{\tau \in \scriptD_k}\|f_\tau^l\|_s^{2t\theta}\bigr),
\end{align}
where $c_0 = \frac{c_1}{2(1-\theta)}$.  Fix $l \geq 0$.  It remains to obtain geometric decay in $l$ of the sum over $k,\tau$ on the right side of \eqref{E:l outside}.

We begin with the case $l = 0$.  Using H\"older's inequality, $2t\theta > s$, Fubini, and $2t\theta > p$,
\begin{align*}
&\sum_k \sum_{\tau \in \scriptD_k} |\tau|^{-2t\theta(\frac1s-\frac1p)}\|f_\tau^0\|_s^{2t\theta} \leq \sum_k \sum_{\tau \in \scriptD_k} |\tau|^{\frac{2t\theta}p-1}\|f_\tau^0\|_{2t\theta}^{2t\theta}\\
&\qquad = \int \bigl(\sum_{k:2^{-k}<|f|^{-\frac pd}} 2^{-kd(\frac{2t\theta}p-1)}\bigr)|f|^{2t\theta}\, d\xi \sim \int |f|^p\, d\xi \sim 1.
\end{align*}

Now we turn to the case $l \geq 1$.  Since $2t\theta > s$
\begin{align*}
&\sum_k \sum_{\tau \in \scriptD_k} |\tau|^{-2t\theta(\frac1s-\frac1p)}\|f_\tau^l\|_s^{2t\theta} \leq \bigl(\sum_k \sum_{\tau \in \scriptD_k} 2^{dks(\frac1s-\frac1p)}\|f_\tau^l\|_s^s\bigr)^{\frac{2t\theta}s}\\
&\qquad = \bigl(\sum_k 2^{dks(\frac1s-\frac1p)}\int_{\{|f|\sim 2^{l+\frac{kd}p}\}} |f|^s\, d\xi\bigr)^{\frac{2t\theta}s} \sim 2^{-l \frac{(p-s)2t\theta}s}.
\end{align*}
Thus we can sum in $l$ on the right side of \eqref{E:l outside} provided $0 < c_1 < \frac{2(p-s)t\theta}s$.  
\end{proof}


\section{Frequency localization}


By Proposition~\ref{P:freq nibbles}, for each $0 < A \leq A_p$ and $\eps>0$, there exists an integer $J$ such that for all $0 \neq f \in L^p$ with $\|\scriptE f\|_q \geq A \|f\|_p$, $f = \sum_{j=1}^J g^j + r^J$, where the remainder $r^J$ contributes an $\eps$ portion of the extension, $\|\scriptE r^J\|_q \leq \eps\|f\|_p$, and the $g^j$ are $\eps$-well-adapted to dyadic cubes $\tau^j$, in the sense that $\supp g^j \subseteq \eps^{-1}\tau^j$ and $|g^j| \leq \eps^{-1}|\tau^j|^{-\frac1p}\|f\|_p$.  Our next task is to control these cubes.  The following proposition states that in the special case $A = A_p$, we can take $j$ to be 1, and the cube to be independent of $\eps$.   (This is easily seen to be false for other values of $A$.)

We recall from the introduction that $\tilde \scriptS_p$ denotes the subgroup of the group $\scriptS_p$ of symmetries of $\scriptE:L^p_\xi \to L^q_{t,x}$ generated by the dilations, the frequency translations, and the space-time translations.    

\begin{proposition} \label{P:freq localization}
For each $\eps > 0$, there exist $\delta > 0$ and $R < \infty$ such that for all $f \in L^p$ satisfying $\|\scriptE f\|_q > (A_p-\delta)\|f\|_p$, there exists a symmetry $S \in \tilde\scriptS_p$ such that 
$$
\|S f\|_{L^p(\{|\xi|>R\} \cup \{|Sf|>R\|f\|_p\})} < \eps \|f\|_p.
$$
The symmetry $S$ may be chosen to depend only on $f$, and not on $\eps$.
\end{proposition}

\begin{proof}[Proof of Proposition~\ref{P:freq localization}]
We begin with the \textit{post hoc} deduction of the independence of the symmetry from $\eps$.  We fix a function $f\in L^p_\xi$, which we may assume has $\|f\|_p = 1$ and $\|\scriptE f\|_p \geq \tfrac12A_p$.  By Lemma~\ref{L:big bite}, there exists a dyadic cube $\tau$ with 
\begin{equation} \label{E:f0}
\|f_0\|_p := \|f \chi_\tau \chi_{\{|f| \lesssim |\tau|^{-\frac1p}\}}\|_p \gtrsim 1.
\end{equation}
By applying a symmetry to $f$, we may assume that $\tau$ is the unit cube.  Now suppose that another symmetry, $Sf$ were to satisfy
\begin{equation} \label{E:f on X}
\|Sf\|_{L^p(\{|\xi|>R\} \cup \{|Sf|>R\})} < \eps,
\end{equation}
for some sufficiently small $\eps$.  We will show that that \eqref{E:f on X} holds with $S$ equal to the identity and $R$ replaced by some slightly larger $R'$ (given below).    As modulations leave the absolute value invariant, we may assume that $Sf = \lambda^{\frac dp}f(\lambda(\cdot-\xi_0))$.  Since $|Sf_0| \leq |Sf|$,
\begin{equation} \label{E:f0 on X}
\|f_0\|_{L^p(\{|\xi+\lambda\xi_0|>\lambda R\} \cup \{|f_0|>\lambda^{-\frac dp}R\})} = \|Sf_0\|_{L^p(\{|\xi|>R\} \cup \{|Sf_0|>R\})} < \eps.
\end{equation}
Hence by \eqref{E:f0} and the triangle inequality, $\|f_0\|_{L^p(\bigcap_{j=1}^4 E_j)} \sim 1$, where 
$$
E_1:=\tau, \quad E_2:=\{|f|\lesssim 1\}, \quad E_3:=\{|\xi+\lambda\xi_0|<\lambda R\}, \quad E_4:=\{|f|<\lambda^{-\frac dp}R\}.
$$
By H\"older's inequality, $\|f_0\|_{L^p(E_1 \cap E_4)} \lesssim \lambda^{-\frac dp}R$ and $\|f_0\|_{L^p(E_2 \cap E_3)} \lesssim (\lambda R)^{\frac1p}$, so $R^{-1} \lesssim \lambda \lesssim R^{\frac pd}$.  Since $E_1 \cap E_3 \neq\emptyset$, $|\xi_0| \lesssim R+\lambda^{-1}$.  Therefore
$$
\{|\xi+\lambda\xi_0|>\lambda R\} \supseteq \{|\xi|> CR^{\frac pd+1}\} \qtq{and} \{|f|>\lambda^{-\frac dp}R\} \supseteq \{|f|>CR^{\frac dp+1}\}.
$$
Inequality \eqref{E:f on X} (see also \eqref{E:f0 on X}) thus implies that
$$
\|f\|_{L^p(\{|\xi|>R'\} \cup \{|f|>R'\})},
$$
with $R' = CR^{\frac dp+\frac pd + 1}$.  

Now we turn to the proof of the main conclusion of the proposition.  

Were the proposition to fail, there would exist $\eps > 0$ and a sequence $\{f_n\} \subseteq L^p_\xi$, satisfying $\|f_n\|_{L^p_\xi} \equiv 1$ and $\|\scriptE f_n\|_{L^q_{t,x}} \to A_p$, but such that 
\begin{equation} \label{E:fn not local}
\liminf \|S_n f_n\|_{L^p(\{|\xi|>n\} \cup \{|S_nf_n|>n\})} > \eps,
\end{equation}
for every sequence $\{S_n\} \subseteq \tilde\scriptS_p$ of symmetries of $\scriptE$.

By Proposition~\ref{P:freq nibbles}, there exist $J \in \N$ and dyadic cubes $\tau_n^j$, $n \in \N$, $1 \leq j \leq J$, such that if we inductively define 
$$
r_n^0:=f_n, \quad g_n^j:=r_n^{j-1} \chi_{\tau_n^j}\chi_{\{|f_n|<C\rho_j^{-1}|\tau_n^j|^{-\frac1p}\}}, \quad r_n^j:=r_n^{j-1}-g_n^j, \quad 1 \leq j \leq J,
$$
then for all $n$ and all functions $|h_n^J| \leq |r_n^J|$,
$$
\|\scriptE h_n^J\|_q < \rho,
$$
with $\rho$ to be determined in a moment.  

Let $F_n:=f_n-r_n^J$.  By our hypotheses and the disjointness of the supports of $F_n$ and $r_n^J$,
\begin{align*}
A_p-\rho& \leq \liminf \|\scriptE F_n\|_q \leq \liminf A_p\|F_n\|_p \leq \liminf A_p(1-\|r_n^J\|_p^p)^{\frac1p} \\ & \leq A_p-c_p \limsup \|r_n^J\|_p^p,
\end{align*}
so, after passing to a subsequence, we may assume that for all $n$, $\|r_n^J\|_p \lesssim \rho^{\frac1p}$, which for $\rho$ sufficiently small implies
\begin{equation} \label{E:fn-Fn}
\|f_n-F_n\|_p < \tfrac\eps2.  
\end{equation}

Since $(f_n)$ is extremizing, for each sufficiently large $n$, $\|g_n^1\|_p \gtrsim 1$.  Indeed, that $\|g_n^1\|_p \ll 1$ implies $\|\scriptE f_n\|_q \ll 1$ follows from the proof of Proposition~\ref{P:freq nibbles}.  (See \eqref{E:max tau tau'}, in particular.)   Applying symmetries if needed, we may assume that $\tau_n^1 = [0,1]^d$ for all $n$.  The remaining cubes may be written
$$
\tau_n^j = \xi_n^j + 2^{-k_n^j}[0,1]^d, \qquad n \in \N, \quad 1 \leq j \leq J,
$$
with $k_n^j \in \Z$ and $\xi_n^j \in 2^{-k_n^j}\Z^d$.  

Passing to a subsequence, we may assume that for each $j$, either $k_n^j$ remains bounded or $|k_n^j| \to \infty$ and that either $\xi_n^j$ remains bounded or $|\xi_n^j| \to \infty$.  Since our $\tau_n^j$ are dyadic (and $j \leq J < \infty$), if $k_n^j$ and $\xi_n^j$ both remain bounded, after passing to a further subsequence, they are constant in $n$.  We say that an index $1 \leq j \leq J$ is good if the parameters $k_n^j$ and $\xi_n^j$ are constant in $n$, and that it is bad otherwise.  We decompose 
$$
F_n = G_n+B_n, \qquad G_n := \sum_{j \: \text{good}} g_n^j, \qquad B_n:=F_n-G_n.
$$
It follows from our hypothesis \eqref{E:fn not local} and the estimate \eqref{E:fn-Fn} that $\liminf \|B_n\|_p > \tfrac\eps2$, so, after passing to a subsequence, 
\begin{equation} \label{E:Gn small}
\|G_n\|_p \leq (1-(\eps/2)^p)^{\frac1p} \leq 1-c\eps^p.
\end{equation}
Since $1 \lesssim \|g_n^1\|_p\leq\|G_n\|_p$, 
\begin{equation} \label{E:Bn small}
\|B_n\|_p = \bigl(\|F_n\|_p^p-\|G_n\|_p^p\bigr)^{\frac1p}  \leq 1-c_0,
\end{equation}
for some $c_0 \gtrsim 1$.  

We claim that after passing to a subsequence, $(\scriptE B_n)$ converges to zero a.e.  Indeed, $B_n=\sum_{bad\,j \leq J} g_n^j$, so it suffices to prove that a subsequence of each bad $\scriptE g_n^j$ tends to zero a.e., as $n \to \infty$.  If $k_n^j \to \infty$, then $g_n^j \to 0$ in $L^1_\xi$, so $\scriptE g_n^j \to 0$ uniformly.  If $k_n^j \to -\infty$, then $g_n^j \to 0$ in $L^{p_0}_\xi$, so $\scriptE g_n^j \to 0$ in $L^{q_0}_{t,x}$.  In the remaining case, $k_n^j$ is bounded, but $|\xi_n^j| \to \infty$.  Thus $g_n^j$ is bounded in $L^2_\xi$ and $\scriptE g_n^j \rightharpoonup 0$ weakly.  By the Rellich--Kondrashov compactness theorem and the local smoothing estimate \cite{ConstantinSautJAMS88, SjolinDuke87, VegaPAMS88}
$$
\iint |(|\nabla_x|^{\frac12}+|\partial_t|^{\frac14}) \scriptE g_n^j(t,x)|^2 |\phi(t,x)|\, dt\, dx \lesssim_\phi \|g_n^j\|_2 \lesssim 1, \qquad \phi \in \scriptS(\R^{1+d}),
$$
a subsequence of $\scriptE g_n^j$ converges to some function $H$ in $L^2_{\rm{loc}}$.  As $\scriptE B_n$ converges weakly to zero, $H \equiv 0$. 

By a result of Br\'ezis--Lieb \cite{BL}, the a.e.\ convergence to zero of $(\scriptE B_n)$ implies that
$$
\lim_{n \to \infty} \|\scriptE F_n\|_q^q - \|\scriptE G_n\|_q^q - \|\scriptE B_n\|_q^q = 0.
$$
Thus by \eqref{E:Gn small}, our hypothesis that $(f_n)$ is an $L^p_\xi$-normalized extremizing sequence, \eqref{E:fn-Fn}, \eqref{E:Gn small}, \eqref{E:Bn small}, and the fact that $q > p$,
\begin{align*}
A_p-\rho &\leq \liminf \|\scriptE F_n\|_q = \liminf \bigl(\|\scriptE G_n\|_q^q + \|\scriptE B_n\|_q^q\bigr)^{\frac1q} \\
& \leq A_p \bigl(\|G_n\|_p^q + \|B_n\|_p^q\bigr)^{\frac1q} \leq A_p \max\{1-c\eps^p,1-c_0\}^{1-\frac pq}\|F_n\|_p^{\frac1q} \\
&\leq A_p \max\{1-c\eps^p,1-c_0\}^{1-\frac pq}.
\end{align*}
Choosing $\rho$ sufficiently small gives a contradiction.  (This approach via Br\'ezis--Lieb and local smoothing is due to Killip--Vi\c{s}an, \cite{ClayNotes}.)
\end{proof}


\section{Space-time localization}


In the previous sections, we used the bilinear theory to prove that near-extremizers have good frequency localization modulo symmetries.  In this section, we take a first step toward localization in spacetime by applying the $L^2_\xi$ theory to prove an $L^p_\xi$-based profile decomposition for frequency localized sequences.  

\begin{proposition} \label{P:Lp profile}
Let $R > 0$ and let $(f_n)$ be a sequence of measurable functions, supported on $\{|\xi|<R\}$, and satisfying $|f_n| < R$.  After passing to a subsequence, there exist $J_0 \in \N \cup \{\infty\}$, $(t_n^j,x_n^j) \in \R^{1+d}$, bounded, measurable functions $\phi^j$ supported on $\{|\xi|< R\}$, and remainders $r_n^J$ such that for each $J < J_0$, $f_n = \sum_{j=1}^J e^{i(t_n^j,x_n^j)(|\xi|^2,\xi)}\phi^j + r_n^J$, and
\begin{itemize}
\item[\rm{(i)}]  For all $j \neq j'$, $\lim_{n \to \infty} \bigl(|t_n^j-t_n^{j'}| + |x_n^j-x_n^{j'}|\bigr) = \infty$,
\item[\rm{(ii)}]  If $\tilde p:=\max\{p,p'\}$, then $\liminf_{n \to \infty}\bigl(\|f_n\|_p - \bigl(\sum_{j=1}^{J_0}\|\phi^j\|_p^{\tilde p}\bigr)^{\frac1{\tilde p}}\bigr) \geq 0$,  
\item[\rm{(iii)}] For $J < J_0$, $\lim_{n \to \infty} \bigl(\|\scriptE f_n\|_q^q - \sum_{j=1}^J \|\scriptE \phi^j\|_q^q - \|\scriptE r_n^J\|_q^q\bigr) = 0$,
\item[\rm{(iv)}] The extensions of the errors tend to zero:  $\lim_{J \to J_0} \limsup_{n \to \infty} \|\scriptE r_n^J\|_q = 0$,
\item[\rm{(v)}]  For all $j$, $\phi^j = \wklim_{n \to \infty} e^{-i(t_n^j,x_n^j)(|\xi|^2,\xi)}f_n$.
\end{itemize}
\end{proposition}

Before beginning the proof, we recall the $L^2_\xi$-based profile decomposition for $\scriptE$.  

\begin{theorem}[\cite{BV, CarlesKeraani, MerleVega}] \label{T:L2 profile}
Let $\{f_n\}$ be a bounded sequence in $L^2_\xi$.  After passing to a subsequence, there exist $J_0 \in \N \cup \{\infty\}$, symmetries $S_n^j \in \tilde\scriptS_2$, nonzero profiles $\phi^j \in L^2_\xi$, and errors $r_n^J \in L^2_\xi$, such that for each $J < J_0$, $f_n = \sum_{j=1}^J S_n^j \phi^j + r_n^J$, and
\begin{itemize}
\item[\emph{(i)}] For all $j \neq j'$, $(S_n^j)^{-1}S_n^{j'} \rightharpoonup 0$ in the weak operator topology,
\item[\emph{(ii)}] For $J < J_0$, $\lim_{n \to \infty} \bigl(\|f_n\|_2^2 - \sum_{j=1}^J \|\phi^j\|_2^2 - \|r_n^J\|_2^2\bigr) = 0$
\item[\emph{(iii)}] For $J < J_0$, $\lim_{n \to \infty} \bigl(\|\scriptE f_n\|_{q_2}^{q_2} - \sum_{j=1}^J \|\scriptE \phi^j\|_{q_2}^{q_2} - \|\scriptE r_n^J\|_{q_2}^{q_2}\bigr) = 0$,
\item[\emph{(iv)}] For all $j$, $\phi^j = \wklim (S_n^j)^{-1} f_n$, 
\item[\emph{(v)}]  The extensions of the errors tend to zero:  $\lim_{J \to J_0} \limsup_{n \to \infty} \|\scriptE r_n^J\|_{q_2} = 0$.
\end{itemize}
\end{theorem}

In proving Proposition~\ref{P:Lp profile}, we may assume that $p \neq 2$.  Let $p_2 := 2$, and choose some $(p_1,q_1)$ at which \eqref{E:extn and rest} holds, such that $p$ lies strictly between $p_1$ and $p_2$.  Set $q_i:=\tfrac{d+2}d p_i$, $i=1,2$.  

The main difficulty is in proving (ii), for which we will use the following technical lemma.   

\begin{lemma} \label{L:Schrodinger projections}
Let $\phi,\psi$ be smooth, compactly supported functions on $\R^d$, with $0 \leq \phi,\psi \leq 1$ and $\phi(0) = \int\psi = 1$.  Let $\{(t_n^j,x_n^j):n \in \N, j \in \N\} \subseteq \R^{1+d}$, with $\lim_{n \to \infty}|(t_n^j-t_n^{j'},x_n^j-x_n^{j'})| = \infty$, for all $j \neq j'$.  For $j \in \N$, define an operator
$$
\pi_n^j f(\xi) := e^{i(t_n^j,x_n^j)(|\xi|^2,\xi)} [\psi(\eta) *_\eta(\phi(\eta)e^{-i(t_n^j,x_n^j)(|\eta|^2,\eta)}f(\eta))],
$$
and for $J \in \N$, define a vector-valued operator $\Pi_n^J := (\pi_n^j)_{j=1}^J$.  Then for each $J$, 
\begin{equation} \label{E:Schrodinger projections}
\limsup_{n \to \infty} \|\Pi_n^J\|_{L^p_\xi \to \ell^{\tilde p}_j(L^p_\xi)} \leq 1,  \qquad \tilde p := \max\{p,p'\}.  
\end{equation}
\end{lemma}

\begin{proof}[Proof of Lemma~\ref{L:Schrodinger projections}]
For each $j$, $n$, and $1 \leq p \leq \infty$, $\pi_n^j$ is a bounded operator on $L^p_\xi$ with norm at most 1, so \eqref{E:Schrodinger projections} is elementary for $p=1,\infty$.  By (complex) interpolation, this leaves us to prove the inequality in the case $p=2$.  By duality, it suffices to prove that $\limsup_{n \to \infty} \|(\Pi_n^J)^*\|_{\ell^2_j L^2_\xi \to L^2_\xi} \leq 1$.  We write
$$
\|(\Pi_n^J)^* \vec f\|_{L^2_\xi}^2 = \sum_j \int |(\pi_n^j)^* f_j|^2 + \sum_{j \neq j'} \int (\pi_n^{j'})^* f_{j'}\overline{(\pi_n^j)^* f_j}\, d\xi.
$$
It is elementary to bound the first term by $\sum_j \|f_j\|_2^2$, so it remains to prove that $\|\pi_n^{j'}(\pi_n^j)^*\|_{L^2_\xi \to L^2_\xi} \to 0$.  Abusing notation slightly, it thus suffices to prove that the sequence $(T_n)$, defined by
$$
T_n g := \psi*_\zeta (e^{i(t_n,x_n)(|\zeta|^2,\zeta)} \phi(\zeta) \psi*g(\zeta)), \qquad g \in L^2_\xi,
$$
tends to zero in $\scriptL(L^2,L^2)$, whenever $|t_n|+|x_n| \to \infty$.  By the support condition on $\phi, \psi$ and stationary phase,
$$
\|T_n g\|_2 \lesssim \|T_n g\|_\infty \lesssim (1+|(t_n,x_n)|)^{-\frac d2} \|\psi*g\|_{C^2} \lesssim (1+|(t_n,x_n)|)^{-\frac d2}\|g\|_2,
$$
whence $\|T_n\|_{2 \to 2} \lesssim (1+|(t_n,x_n)|)^{-\frac d2} \to 0$.  
\end{proof}

\begin{proof}[Proof of Proposition~\ref{P:Lp profile}]
As the sequence $(f_n)$ is bounded in $L^2$ (albeit with an $R$-dependent bound), we may apply the profile decomposition in Theorem~\ref{T:L2 profile}.  Each symmetry $S_n^j$ arising therein may be written as a composition of a dilation with parameter $\lambda_n^j$, a frequency translation with parameter $\xi_n^j$, and a spacetime translation with parameter $(t_n^j,x_n^j)$.  By the size and support conditions on the $f_n$, as well as the definition of the $\phi^j$ and their nontriviality, the dilation parameters are bounded away from $0$ and $\infty$, and the frequency parameters are bounded.  Thus, after passing to a subsequence, for each $j$ the dilations and frequency translations converge in the strong operator topology.  Putting the limit on the $\phi^j$ if needed, we may assume that
$$
S_n^j \phi^j = e^{i(t_n^j,x_n^j)(|\xi|^2,\xi)}\phi^j.
$$
Conclusions (i) and (v) follow.

Conclusion (iii) follows from local smoothing and the Br\'ezis--Lieb inequality as in the proof of Proposition~\ref{P:freq localization}.  Of course, $f_n$ is bounded in $L^{p_1}_\xi$, and Br\'ezis--Lieb also yields (iii) with $q$ replaced by $q_1$.  Thus, after passing to a subsequence, $\|\scriptE r_n^J\|_{q_1}$ is uniformly bounded for all $n$ and $J$.  We already know that 
$$\lim_{J \to \infty} \limsup_{n \to \infty} \|\scriptE r_n^J\|_{q_2} = 0,$$
 so (iv) follows from H\"older's inequality.  

This leaves us to prove (ii).  Fix $J \in \N$ with $J \leq J_0$.  Choose smooth, compactly supported $\psi,\phi$ with $0 \leq \phi,\psi \leq 1$ and  $\phi(0) = \int \psi = 1$ and $\|\psi*(\phi \phi^j) - \phi^j\|_p < \eps$.  We claim that 
$$
\lim_n \|\pi_n^j f_n - e^{i(t_n^j,x_n^j)(|\xi|^2,\xi)} \psi * (\phi \phi^j)\|_p = 0,
$$
where $\pi_n^j$ is defined as in Lemma~\ref{L:Schrodinger projections}.  To this end, it suffices to prove that for all $1 \leq j \neq j' \leq J$,
\begin{equation} \label{E:orthog projections}
\lim_{n \to \infty} \|\pi_n^j(e^{i(t_n^{j'},x_n^{j'})(|\xi|^2,\xi)}\phi^{j'})\|_p = 0 \qtq{and} \lim_{n \to \infty} \|\pi_n^j r_n^j\|_p = 0.
\end{equation}
By (v), the claimed limits amount to proving that $\lim \|\psi*(\phi g_n)\|_p = 0$, whenever $(g_n)$ is a sequence in $L^p$ converging weakly to zero.  This is an immediate consequence of the Dominated Convergence Theorem and the compact support of $\phi,\psi$.  We send $\eps \searrow 0$, and (ii) follows from Lemma~\ref{L:Schrodinger projections}.  
\end{proof}


\section{Proof of Theorem~\ref{T:main}} \label{S:prove profile}


Let $(f_n)$ be an $L^p$-normalized extremizing sequence.  By Proposition~\ref{P:freq localization}, after applying a symmetry, 
$$
\|\scriptE f_n^R\|_q \geq A_p - \eps(n,R),
$$
where $f_n^R := f \chi_{(\{|\xi| < R\} \cup \{|f| < R\})}$, and $\lim_{R \to \infty} \limsup_{n \to \infty} \eps(n,R) = 0$.  We consider the integer truncations, $f_n^R$ with $R=m \in \N$.  By Proposition~\ref{P:Lp profile}, after passing to a subsequence in $n$ (which is independent of $m$), we may decompose 
$$
f_n^m:= \sum_{j=1}^J e^{i(t_n^{m,j},x_n^{m,j})(|\xi|^2,\xi)}\phi^{m,j} + r_n^{m,j}, \qquad 1 \leq J < J_0 \in \N \cup \{\infty\},
$$
where the decomposition on the right satisfies the conclusions of that proposition.  

By conclusions (iii) and (iv), then (ii) of Proposition~\ref{P:Lp profile}, and $q > \tilde p$,
\begin{align*}
&A_p^q - o_m(1) \leq \limsup_{n \to \infty} \|\scriptE f_n^m\|_q^q = 
\sum_{j=1}^{J_0}\|\scriptE \phi^{m,j}\|_q^q 
\leq A_p^{\tilde p} \max_{j\leq J_0} \|\scriptE \phi^{m,j}\|_q^{q-\tilde p} \sum_{j=1}^{J_0}\|\phi^{m,j}\|_p^{\tilde p} \\
&\qquad \leq A_p^{\tilde p} \max_{j\leq J_0} \|\scriptE \phi^{m,j}\|_q^{q-\tilde p} \leq A_p^q\max_{j\leq J_0} \|\phi^{m,j}\|_p^{q-\tilde p}.
\end{align*}
Choose $j=j_m$ to maximize $\|\scriptE \phi^{m,j}\|_q$, and set 
$$
\phi^m := \phi^{m,j_m}, \qquad (t_n^m,x_n^m):=(t_n^{m,j_m},x_n^{m,j_m}).
$$
Then
\begin{equation} \label{E:phi, Ephi big}
1-o_m(1) \leq \|\phi^{m}\|_p \leq 1, \qtq{and} A_p-o_m(1) \leq \|\scriptE \phi^{m}\|_p.
\end{equation}
Since 
\begin{gather}\notag
e^{-i(t_n^m,x_n^m)(|\xi|^2,\xi)} f_n^m \rightharpoonup \phi^m, \qtq{weakly in $L^p_\xi$,}\\\label{E:phim big}
\qtq{and} \|\phi^m\|_p \geq (1-o_m(1)) \limsup_{n \to \infty} \|f_n^m\|_p, \qtq{as $m \to \infty$,}
\end{gather}
strict convexity of $L^p$ implies
$$
\limsup_{n \to \infty} \|f_n^m - e^{i(t_n^m,x_n^m)(|\xi|^2,\xi)} \phi^m\|_p = o_m(1), \: \text{as $m \to \infty$}.
$$
(See Theorem 2.5 and the proof of Theorem 2.11 in \cite{LiebLoss}.)  

By Proposition~\ref{P:freq localization} and the triangle inequality,
$$
\limsup_{n \to \infty} \|f_n-e^{i(t_n^m,x_n^m)(|\xi|^2,\xi)}\phi^m\|_p = o_m(1),\: \text{as $m \to \infty$},
$$
whence
$$
\limsup_{n \to \infty} \|e^{i(t_n^{m'},x_n^{m'})(|\xi|^2,\xi)}\phi^{m'} - e^{i(t_n^m,x_n^m)(|\xi|^2,\xi)}\phi^m\|_p =  o_{\min\{m,m'\}}(1).
$$
Applying the projection $\pi_n^m$ and using \eqref{E:orthog projections} and \eqref{E:phim big}, for sufficiently large $m,m'$, $|(t_n^m-t_n^{m'},x_n^m-x_n^{m'})|$ remains bounded as $n \to \infty$.  Applying a spacetime modulation to $f_n$, we may assume that $(t_n^M,x_n^M) \equiv 0$, for some fixed, sufficiently large $M$.  Passing to a subsequence, we may thus assume that $(t_n^m,x_n^m) \to (t^m,x^m)$ for all $m \geq M$, whence, replacing $\phi^m$ with $e^{i(t^m,x^m)(|\xi|^2,\xi)}\phi^m$, we may assume that $(t^m_n,x^m_n) \to 0$ for all $m \geq M$.  Thus $f \mapsto e^{i(t^m_n,x^m_n)(|\xi|^2,\xi)}f$ converges to the identity in the strong operator topology on ${\rm{Aut}}(L^p_\xi)$, for all $m$.  In summary, we knew that
$$
\limsup_{n \to \infty} \|f_n - e^{i(t_n^m,x_n^m)(|\xi|^2,\xi)}\phi^m\|_p = o_m(1), \qtq{as} m \to \infty;
$$
we now know that
\begin{equation} \label{E:fn to phim}
\limsup_{n \to \infty} \|f_n-\phi^m\|_p = o_m(1), \qtq{as} m\to\infty.
\end{equation}
By \eqref{E:fn to phim} and the triangle inequality, $\phi^m$ is Cauchy, hence convergent, in $L^p_\xi$ as $m \to \infty$, and $\{f_n\}$ converges to the limit, which is an extremizer, as $n \to \infty$.  

\section{Proof of the corollary:  Extremizers for the restriction operator}\label{S:corollary}

If $\{g_n\}$ is an $L^{q'}_{t,x}$-normalized extremizing sequence for the restriction operator $\scriptR$, by duality, $f_n:=|\scriptR g_n|^{p'-2}\scriptR g_n \in L^p_\xi$ is extremizing for $\scriptE$, with $\|f_n\|_p \to A_p^{p'-1}$.  By Theorem~\ref{T:main}, after passing to a subsequence, there exist extension symmetries $S_n \in \tilde\scriptS_p$ such that $S_nf_n$ converges in $L^p_\xi$ to an extension extremizer $f$.  As $S_n f_n = |\scriptR T_n g_n|^{p'-2} \scriptR T_n g_n$, for a corresponding sequence $\{T_n\}$ of restriction symmetries, we may assume, replacing $g_n$ with $T_n g_n$, that $f_n \to f$ in $L^p_\xi$.  Passing to a subsequence, $\{g_n\}$, being bounded, has a weak limit:  $g_n \rightharpoonup g$ in $L^{q'}_{t,x}$.  We claim that $g$ is a restriction extremizer and that this weak convergence is in fact strong.  Indeed,
\begin{align*}
A_p^{p'}\|g\|_{q'} = A_p \|g\|_{q'}\|f\|_p \geq |\langle g, \scriptE f \rangle| = \lim|\langle g_n,\scriptE f_n \rangle| = \lim\|\scriptR g\|_{p'}^{p'} = A_p^{p'}.
\end{align*}
By Theorem 2.11 of \cite{LiebLoss}, weak convergence combined with convergence of norms implies strong convergence, $g_n \to g$ in $L^{q'}_{t,x}$.  By continuity of $\scriptR$, it follows that $g$ is a restriction extremizer.



\end{document}